\documentclass{amsart}
\usepackage[utf8]{inputenc}

\usepackage{algorithm}
\usepackage{algpseudocode}
\usepackage{mathtools}

\usepackage{xcolor}

\usepackage{hyperref}
\hypersetup{colorlinks=true, citecolor=teal, linkcolor=blue}

\newcommand{\CC}{\mathbb{C}}
\newcommand{\FF}{\mathbb{F}}
\usepackage{fullpage}

\usepackage{commath, graphicx}

\usepackage{amsmath, amsfonts, amsthm, amssymb, amscd, multirow, tikz}

\theoremstyle{plain}
\newtheorem{theorem}{Theorem}[section]
\newtheorem{proposition}[theorem]{Proposition}
\newtheorem{lemma}[theorem]{Lemma}

\theoremstyle{definition}
\newtheorem{definition}[theorem]{Definition}

\theoremstyle{remark}
\newtheorem{remark}[theorem]{Remark}
\newtheorem{example}[theorem]{Example}

\theoremstyle{remark}

\newcommand{\M}{\mathsf{M}}
\DeclareMathOperator{\Aut}{Aut}
\DeclareMathOperator{\disc}{disc}
\newcommand{\QQ}{\mathbb{Q}}
\newcommand{\OO}{\mathcal{O}}
\newcommand{\ZZ}{\mathbb{Z}}

\newcommand{\Fq}{\bb{F}_q}

\newcommand{\cc}{\bb{C}}
\newcommand{\zz}{\bb{Z}}

\newcommand{\wt}{\widetilde}

\newcommand{\e}{\varepsilon}
\newcommand{\bb}{\mathbb}
\newcommand{\mc}{\mathcal}

\newcommand{\ra}{\rightarrow}

\newcommand{\f}{\mathfrak}

\newcommand{\llp}{\left(}
\newcommand{\rrp}{\right)}

\newcommand{\sus}{\subset}

\usepackage{bookmark}

\DeclareMathOperator{\Char}{\operatorname{char}}

\renewcommand{\del}{\partial}

\newcommand{\ol}{\overline}

\renewcommand{\P}{\mathfrak{P}}

\DeclareMathOperator{\End}{End}
\DeclareMathOperator{\Gal}{Gal}

\usepackage{xcolor}

\title{Computing Isogenies at Singular Points of the Modular Polynomial}
\author{William E. Mahaney and Travis Morrison}
\date{}

\begin{document}
\begin{abstract}
Suppose  $(j_1, j_2)$ is a $k$-rational (possibly singular) zero of the $\ell$-th classical modular polynomial $\Phi_{\ell}(X,Y)$, where  $j_1, j_2 \neq 0, 1728$ and $k$ is either a number field or a finite field of characteristic $p > 4\ell$. We give an algorithm that computes all normalized $\ell$-isogenies $\phi\colon E_1\to E_2$ where $j(E_i)=j_i$. This extends an algorithm of Elkies that requires $(j_1,j_2)$ to be a nonsingular point. We also give an $\widetilde{O}(\ell^3)$ algorithm for computing all normalized $\ell$-isogenies of a given elliptic curve $E$, provided neither $E$ nor any curve $\ell$-isogenous to it has $j$-invariant $0$ or $1728$. 

\end{abstract}
\maketitle
\;

\section{Introduction}
For a pair of elliptic curves $E_1$ and $E_2$ defined over a field $k$, a separable isogeny $\phi\colon E_1\to E_2$ is determined (up to composition with an isomorphism of the target) by its kernel. 
This result can be made effective: for example, V\'{e}lu~\cite{velu71} gives formulas for computing $\phi$ (as a rational map) given the points in $\ker\phi$ and Kohel~\cite[Section 2.4]{Koh96} gives formulas for computing this rational map from the {\em kernel polynomial} of $\phi$, the monic polynomial that vanishes precisely at the $x$-coordinates of the nonzero points in $\ker\phi$. The methods of V\'{e}lu and Kohel also give an equation for the codomain $E_2$ of $\phi$, and the resulting isogeny is always {\em normalized}: the invariant differential $dx/2y$ of $E_2$ pulls back under $\phi$ to the invariant differential $dx/2y$ of $E_1$. The isogeny $\phi$ will be defined over $k$ if $\ker\phi$ is $\Gal(\overline{k}/k)$-stable, or equivalently if the kernel polynomial of $\phi$ has coefficients in $k$. 

Alternatively, given $E_1$, $E_2$, and $d$ such that there exists a cyclic, normalized $d$-isogeny $\phi \colon E_1 \to E_2$, this isogeny can be computed with Stark's algorithm~\cite{Stark73} and even more efficiently with an algorithm due to Bostan, Morain, Salvy, and Schost~\cite{BMSS}. If one only knows there exists an isogeny from $E_1$ of prime degree $\ell$ to a curve of $j$-invariant $j_2$, perhaps because one has found that $\Phi_{\ell}(j(E_1),j_2)=0$ where $\Phi_{\ell}(X,Y)$ is the {\em modular polynomial}, an algorithm of Elkies~\cite[Section 3]{Elkies} can compute a curve $E_2$ for which a normalized $\ell$-isogeny $\phi \colon E_1 \to E_2$ exists and subsequently the kernel polynomial for $\phi$, provided $(j(E_1), j(E_2))$ is a nonsingular root of $\Phi_\ell$ and the characteristic of $k$ is either zero or sufficiently large relative to $\ell$.
\footnote{See~\cite[Section 7]{Schoof} for further details and~\cite[Section 25.2.1]
{Galbraith} for pseudocode.} This algorithm is an 
important subroutine in the SEA algorithm for point-counting on elliptic curves over finite fields and in the ``Elkies step'' in the Couveignes--Rostovtsev--Stolbunov key exchange protocol~\cite{MR3897882}, for example. 

Elkies' method for computing an equation for the isogenous curve as described in~\cite{Schoof} works as follows. 
Suppose $E_1:y^2=x^3+A_1x+B_1$ is an elliptic curve over $\bb{C}$ which is $\ell$-isogenous to the curve $E_2$.
Then there exists $q_0=\exp(2i \pi \tau_0)\in \bb{C}$ such that $A_1=-G_4(q_0)/48$, $B_1=G_6(q_0)/864$,  $j(E_1)=j(q_0)$, and $j(E_2)=j(q_0^{\ell})$ where $G_4$ and $G_6$ are the Eisenstein series of weights $4$ and $6$ respectively and $j(q)$ is the $j$-function.

 Given a function $f(q)$ let $f'(q)$ denote the derivative of $f(q)=f(q(2i\pi \tau))$ with respect to $2 i \pi \tau$. Explicitly, \(f'(q)=q\frac{df}{dq}.\)

With these conventions
\[
A_2 = \frac{\ell^4}{48} \cdot \frac{j'(q_0^\ell)^2}{j_2 (1728 -j_2 ) }\quad 
\text{and} \quad B_2 = \frac{\ell^6}{864} \cdot \frac{j'(q_0^\ell)^3}{j_2^2 (1728 - j_2)}
\]
yields coefficients for a short Weierstrass equation for $E_2\colon y^2=x^3+A_2x+B_2$ with $j$-invariant $j(q_0^\ell)=j(E_2)=j_2$ such that the $\ell$-isogeny $\phi\colon E_1\to E_2$ is normalized. Therefore one can compute $E_2$ as soon as one obtains 
$j'(q_0^\ell)$. Letting $\Phi_{\ell}(X,Y)$ denote the $\ell$-th modular polynomial parametrizing $\ell$-isogenous $j$-invariant pairs we have $\Phi_{\ell}(j(q),j(q^\ell))$ is identically zero as a $q$-series. Differentiating both sides of the equation $\Phi_{\ell}(j(q),j(q^\ell)) = 0$ with respect to $2i\pi \tau$ yields 
\[
\left(\Phi_{\ell}(j(q),j(q^\ell))\right)'=j'(q)\Phi_X(j(q),j(q^\ell))+\ell j'(q^\ell)\Phi_Y(j(q),j(q^\ell))=0.
\]
So, assuming $(j(E_1),j(E_2))$ is a nonsingular point on $\Phi_\ell(X,Y)=0$, one can solve for $j'(q_0^\ell)$  and hence $E_2$ in terms of the equation for $E_1$ and the $j$-invariant of $E_2$. 

If $(j(E_1),j(E_2))$ is a singular $k$-rational point on $Z_0(\ell): \Phi_{\ell}(X,Y)=0$, there are (at least) two $\ell$-isogenies $\phi,\psi\colon E_1\to E_2$ with distinct kernels. 
Schoof writes that the approach to computing the isogenous curve $E_2$ ``does not work if $\Phi_X(j_1,j_2)=\Phi_Y(j_1,j_2)=0$.'' 
However, the broader idea of computing $j'(q_0^\ell)$ in order to find the model for $E_2$ does. 

In this paper, we take up the problem of computing an $\ell$-isogeny $E_1\to E_2$ when $(j(E_1),j(E_2))$ is a singular point on $\Phi_\ell=0$. To do this, we exploit the fact that the normalization of the singular curve $\Phi_\ell=0$ is $X_0(\ell)$, the smooth curve whose non-cuspidal points are isomorphism classes of pairs $(E,C)$ where $C$ is a cyclic subgroup of $E$ of order $\ell$. The $\ell$-isogenies from $E_1$ to elliptic curves with $j$-invariant $j_2$ correspond to the branches of $\Phi_\ell=0$ at the point $(j_1,j_2)$. These branches are determined by the various tangent directions at a (possibly singular) point, and it turns out that the tangent direction whose branch corresponds to the normalized $\ell$-isogeny $\phi\colon E_1\to E_2$ has tangent direction $(j'(q_0),\ell j'(q_0^\ell))$. Thus the points $[E_1,C]$ on the smooth curve $X_0(\ell)$ above $(j_1,j_2)$ can be explicitly computed from the tangent directions at this point, assuming $j_1,j_2\not\in\{0,1728\}$. Our method computes a polynomial $F_P(t)$ whose roots correspond to the various tangent directions at $P$, and we give explicit formulas to obtain a normalized isogeny $\phi_r$ from a root $r$ in Theorem~\ref{thm:main}. In fact, this approach even works in positive characteristic $p$ if $p$ is sufficiently large relative to $\ell$. We prove 
\begin{theorem}\label{thm:isogeny_FP_root_bijection_intro}
        Assume either $\Char(k)=0$ or $\Char(k)=p > 4 \ell$.  Fix a $k$-rational point $P = (j_1, j_2)$ of multiplicity $m$ on $Z_0(\ell)$ such that $j_1, j_2 \neq 0, 1728$ and fix an elliptic curve $E_1$ over $k$ with $j$-invariant $j_1$. There is an explicit, computable polynomial $F_P(t)\in k[t]$, depending on $E_1$ and $j_2$, whose roots in $\overline{k}$ correspond bijectively to the normalized $\ell$-isogenies from $E_1$ to elliptic curves with $j$-invariant $j_2$. The isogeny corresponding to the root $r$ is defined over $k(r)$. 
\end{theorem}
We prove Theorem~\ref{thm:isogeny_FP_root_bijection_intro} as Theorem~\ref{thm:isogeny_FP_root_bijection} in Section~\ref{sec:normalizedcodomain}. 
The correspondence between roots of $F_P$ and isogenies of $E_1$ in Theorem~\ref{thm:isogeny_FP_root_bijection_intro} is effective. In particular, we prove that a root $r$ of $F_P(t)$ determines the equation of $E_2$ such that the $\ell$-isogeny $\phi_r\colon E_1\to E_2$ is normalized.  With an equation for $E_2$, the algorithm {\rm \sf fastElkies'} of~\cite{BMSS} computes $\phi_r$ efficiently. This allows one to compute all $\ell$-isogenies of $E$ efficiently. 
The algorithm {\rm \sf fastElkies} of ~\cite{BMSS}  computes $\phi$ even faster if one knows, in addition to the model $E_2$, the coefficient of $t^{(\ell-3)/2}$ in the kernel polynomial for $\phi_r$. Theorem~\ref{thm:main} shows how to compute this quantity above a singular point $(j_1,j_2)$ once we have already computed the model $E_2$. 

If $P=(j_1,j_2)$ is a singular $k$-rational point on $\Phi_\ell=0$, the points above it on $X_0(\ell)$ may not be $k$-rational. Our method for computing the normalized codomain $E_2$ obtains $j'(q_0^\ell)$ as a root of a polynomial $F_P(t)\in k[t]$ of degree $m$, where $m$ is the multiplicity of the point $P$ on $Z_0(\ell)$. The corresponding isogeny is $k$-rational if the root is rational. Theorem~\ref{thm:irrationalisogeniesabovemultipoint} shows that if $P\in Z_0(\ell)(k)$ does not lift to a point in $X_0(\ell)(k)$, then $P$ is singular. In this case, the polynomial $F_P(t)$ has no roots in $k$. 

Finally, we remark that the number of singular points on $\Phi_\ell=0$ can be bounded. A singular point $(j(E_1),j_2)$ gives rise to a cyclic endomorphism of $E_1$ of degree $\ell^2$ and hence an embedding of an imaginary quadratic order of negative discriminant $t^2-4\ell^2$ into $\End(E_1)$. The (weighted, by the size of the unit group) number of such embeddings is counted by the Hurwitz class number $H(4\ell^2-t^2)$. The Kronecker--Hurwitz class number relation (see~\cite[p.1]{MR429750}, for example) gives $\sum_{|t|<2\ell}H(4\ell^2-t^2)=2\ell^2+\ell+1/6$. Thus there are $O(\ell^2)$ many singular points on $\Phi_\ell=0$ over $\CC$. Every singular point over a finite field is the reduction of one in characteristic $0$ \cite[\S7, p248]{Schoof}. This curve has $q+O(\ell\sqrt q)$ points over the finite field $\FF_q$, so the proportion of singular points is $O(\ell^2/q)$. Thus until $\ell$ is large relative to the field size, singular points are very rare, and this is why this edge case is often ignored, for example in the SEA algorithm. However, there are roughly $p/12$ isomorphism classes of supersingular curves, all with models defined over $\FF_{p^2}$, so the singular, supersingular points on $\Phi_\ell=0$ are not rare once $\ell^2\approx p$.

\subsection{Applications}
We highlight some algorithmic applications of the ideas and formulas in this paper. 

\subsubsection{Computing all $\ell$-Isogenies of an Elliptic Curve} Computing one (or all) $\ell$-isogenies of an elliptic curve is an important subroutine in many algorithms; we detail some below. For this discussion, we focus on the case that $E$ is defined over a finite field $\FF_q$ and let $n=\log q$. Cremona and Watkins describe how to do this efficiently when the genus of $X_0(\ell)$ is $0$ in an unpublished preprint; the method is described in Chapter 4 of Tsukazaki's thesis~\cite{warwick57568}. Tsukazaki also gives an algorithm for arbitrary $\ell$; he does not calculate the complexity, but his algorithm in \cite[Chapter 3]{warwick57568} requires factoring the $\ell$th division polynomial, requiring $\tilde{O}(\ell^4)$ field operations. Our results allow for a quasi-cubic algorithm, requiring only $\tilde{O}(\ell^3)$ bit operations, under the assumption $p >  4\ell$ and neither $E$ nor any of its $\ell$-neighbors has $j$-invariant $0$ or $1728$. 

More precisely, Theorems~\ref{thm:isogeny_FP_root_bijection} and~\ref{thm:main} result in an efficient algorithm for computing all $\ell$-isogenies of $E_1$, provided $\Char(k)=0$ or $\Char(k)>4\ell$ and neither $E_1$ nor any of its $\ell$-isogenous curves have $j$-invariant $0$ or $1728$. We discuss the algorithm and its complexity (when $k$ is a finite field) in Section~\ref{sec:complexity}. Let $\M(n)$ denote the bit-complexity of multiplication of $n$-bit integers. We prove the following as Theorem~\ref{thm:complexity} in Section~\ref{sec:complexity}.
\begin{theorem}\label{thm:complexity_intro}
        Let $k=\FF_{p^e}$ be a finite field with $\log e = O(\log p)$ and let $E$ be an elliptic curve defined over $k$. Let $\ell$ be a prime and assume $p>4\ell$. Assume $j(E)\not=0,1728$ and that $E$ is not $\ell$-isogenous to a curve of $j$-invariant $0$ or $1728$. Let $n=e\log p$. The normalized $\ell$-isogenies of $E$ can be computed in an expected 
    $O(
    \ell^3(\log \ell)^3
    +\ell n\M(\ell n) 
    + \ell^3\M(n)  
    )$
    bit operations.
\end{theorem}
See Theorem~\ref{thm:complexity} for the precise complexity and a description of the algorithm. We implemented a simplified version of this algorithm in SageMath, and our algorithm was later integrated into SageMath by Lorenz Panny, resulting in a substantial speedup over the prior algorithm for computing all $\ell$-isogenies in the ``generic'' case; see \url{https://github.com/sagemath/sage/pull/42436} and \url{https://github.com/sagemath/sage/pull/42456}.

\subsubsection{Point Counting \& Trace Computations}

The Schoof--Elkies--Atkin (SEA) algorithm for computing the number of points on an ordinary elliptic curve $E$ over $\bb{F}_q$ computes kernels of $\ell$-isogenies of $E$ in order to compute the trace of Frobenius modulo small primes $\ell$ and ultimately recover the trace of Frobenius on $E$ along with $\#E(\Fq)$. 
If a singular point $\left( j(E), j_2 \right)$ of $\Phi_{\ell}$ is found, then the formulas of Elkies and Schoof do not apply; we hope our algorithm, which has fewer failure cases, will prove useful. We remark that if $(j(E),j_2)$ is a singular point of $\Phi_\ell$, then $E$ has a cyclic endomorphism $\beta$ of degree $\ell^2$, so the discriminant of $\End(E)$ is small relative to $q$ and the trace of the Frobenius endomorphism is more easily computed with the methods of \cite[Section 4]{Schoof} rather than the full SEA algorithm \cite[p248-249]{Schoof}.

More generally, the  SEA algorithm to compute the trace of Frobenius of an ordinary elliptic curve can be extended to compute the trace of an endomorphism $\alpha$ of a supersingular elliptic curve $E$~\cite{morrison2025seaalgorithmendomorphismssupersingular}. For the application of computing the trace of an endomorphism $\alpha$ of a supersingular curve $E$, a priori the endomorphisms $\alpha$ and $\beta$ arising from a singular point $(j(E),j_2)$ of $\Phi_\ell$ do not commute, so knowledge of $\beta$ need not aid in computing the trace of $\alpha$.

\subsubsection{Orders in the Endomorphism Ring from Singularities} Given a singular point $\left(j(E_1), j(E_2)\right)$ of multiplicity 3 over a finite field $\FF_q$, the elliptic curve $E_1$ is supersingular and our algorithm can be used to compute a pair of non-commuting endomorphisms $\alpha, \beta$ of $E_1$ of degree $\ell^2$. In \cite{EisScul} this pair of non-commuting endomorphisms can be used along with a factorization of the reduced discriminant $\Delta$ of the quaternionic order generated by $1, \alpha, \beta, \alpha \beta$ to apply \cite[Theorem 1.1]{EisScul} and compute $\End(E_1)$. 

\subsubsection{Orientations of Elliptic Curves}
A (proper) $\f{O}$-orientation of a supersingular elliptic curve $E$ is an embedding $\iota: \f{O} \hookrightarrow \End(E)$ of an imaginary quadratic order into the endomorphism ring of $E$ such that $\iota(\f{O})$ is not contained in a larger quadratic order contained in $\End(E)$. Given $\disc(\f{O}) < p$ the authors of \cite{OrientationComputation} prove a subexponential reduction from the computational $\f{O}$-orientation problem (given an $\f{O}$-orientable elliptic curve, compute an orientation) to the decisional $\f{O}$-orientation problem (does a given supersingular elliptic curve have an $\f{O}$-orientation?) by describing an explicit algorithm to compute an $\f{O}$-orientation of an $\f{O}$-orientable curve given access to an oracle for the decisional $\f{O}$-problem. As part of this algorithm, the authors of \cite[Section 2.4]{OrientationComputation} need to compute all $\ell$-isogenies between a supersingular elliptic curve and its $\ell+1$ neighbors. In the singular case the authors factor the division polynomial to enumerate the $\ell+1$ cyclic subgroups of the $\ell$-torsion of $E$ and obtain all isogenies this way, but our algorithm avoids the need for the factoring step and provides an overall speedup to their algorithm.

\subsection{Paper Overview}

In Section~\ref{sec:background}, we  recall relevant definitions and notation. In Section~\ref{sec:singular}, we collect some facts about singular points on the modular polynomial. We then prove our main results in Sections~\ref{sec:normalizedcodomain} and~\ref{sec:abscissasum} on isogenies above a singular point on $\Phi_\ell=0$ that allow us to compute these isogenies efficiently. We sketch an algorithm for computing all $\ell$-isogenies of a curve in Section~\ref{sec:complexity} and prove it terminates in an expected $\widetilde{O}(\ell^3)$ bit operations under some mild assumptions in Theorem~\ref{thm:complexity}. We discuss our results and give several examples of our algorithm for computing isogenies above a singular point in Section~\ref{sec:examples}. Our implementation in SageMath~\cite{sage} is available at \url{https://github.com/wmahaney/ModularMultipoints}.

\section{Background}\label{sec:background}
In this section, we include some background on elliptic curves, modular curves, and singularities of plane curves. We refer the reader to ~\cite{Silverman,DiamondShurman,Fulton} for further background. Throughout this paper we assume the field $k$ denotes either a subfield of $\bb{C}$ or a finite field of characteristic $p \ge 5$. 
\subsection{Elliptic Curves} 
An elliptic curve over $k$ is a smooth projective genus $1$ curve with a $k$-rational point. If the characteristic of $k$ is not $2$ or $3$, then an elliptic curve has a {\em short Weierstrass equation over $k$}, that is, an equation of the form $y^2 =  x^3 + Ax + B$ where $x,y$ are indeterminates and $A, B \in k$ satisfy $4A^3 + 27B^2  \neq 0$. 
Given an elliptic curve $E: y^2 = x^3 + Ax + B$ the {\em  invariant differential} of $E$ is defined as 
\(
\omega_E := \frac{dx}{2y} \in \Omega_E
\)
. The {\em  $j$-invariant} of $E$ is given by  \[
j(E) := \frac{-1728(4A)^3   }{ -16(4A^3 + 27B^2)  }   = \frac{4 \cdot 1728A^3}{4A^3 + 27B^2}.
\]
Two elliptic curves $E_1, E_2$ defined over $k$ are $\ol{k}$-isomorphic if and only if $j(E_1) = j(E_2)$.  Moreover, for every $j\in k$, there is an elliptic curve $E$ defined over $k$ with $j$-invariant $j$. 

Given a pair of elliptic curves $E_1,E_2$ an {\em isogeny} from $E_1 \to E_2$ is a non-constant rational map $\phi\colon E_1 \to E_2$ of projective varieties such that $\phi(\infty_{E_1}) = \infty_{E_2}$.
Any isogeny from $E_1 \to E_2$ induces a group homomorphism $\phi\colon E_1 ( \ol{k} )\to E_2 ( \ol{k} )$. The \textit{degree} of an isogeny is its degree as a rational map. An isogeny $\phi\colon E_1 \to E_2$ between elliptic curves is \textit{separable} 
provided it is separable as a rational map. If $\phi$ is separable, then $\deg\phi =\#\ker\phi$. A separable isogeny is \textit{cyclic} if its kernel is a cyclic subgroup of $E_1(\ol{k})$. 

Two $j$-invariants $j_1, j_2 \in k$ are \textbf{$\ell$-isogenous} if there are a pair of elliptic curves $E_1, E_2$ defined over $\overline{k}$ with j-invariants $j(E_i) = j_i$ such that there exists an $\ell$-isogeny $\phi \colon E_1 \to E_2$.

Every isogeny $\phi$ factors as $\phi=\phi'\circ[n]\circ\pi^{d}$ where $\pi$ is the $p$-power Frobenius isogeny, $p^d$ is the inseparable degree of $\phi$, $[n]$ is the multiplication-by-$n$ map, and $\phi'$ is separable and cyclic. Moreover, every $n$-isogeny $\phi\colon E_1 \to E_2$ can be factored as a composition of prime degree isogenies. Given these last two facts we can restrict our attention to separable prime degree isogenies.

\subsection{The Modular Curve \texorpdfstring{$X_0(\ell)$}{X0l}}
Let $N\geq 1$ be an integer and consider pairs $(E,C)$ where $E$ is an elliptic curve over $\cc$ and $C$ is a cyclic subgroup of $E$ of order $N$. Two such enhanced elliptic curves $(E,C)$ and $(E',C')$ are isomorphic if there exists an isomorphism $u\colon E\to E'$ such that $u(C)=C'$. The isomorphism classes $[(E,C)]$ of such enhanced elliptic curves correspond to the points in $Y_0(N)= \Gamma_0(N)\backslash \mathbb{H}$ whose compactification $X_0(N)$ is a Riemann surface. The function field of the Riemann surface $X_0(N)$, thought of as an algebraic curve, is defined by the $N$th \textbf{classical modular polynomial} $\Phi_N(X,Y)\in \ZZ[X,Y]$. This polynomial has the property $\Phi_N(j_1, j_2) = 0$ if and only if there are elliptic curves $E_{1}$ and $E_{2}$ defined over $\cc$ with $j$-invariants $j_1$ and $j_2$ and a cyclic $N$-isogeny $\phi\colon E_{j_1}\to E_{j_2}$. 

Because $\Phi_N$ has integer coefficients, we can interpret the equation $\Phi_N=0$ and the corresponding smooth projective curve $X_0(N)$ over any field. Over any field in which $N$ is invertible, the non-cuspidal points of $X_0(N)$ correspond to isomorphism classes  $[(E,C)]$ as above, by work of Igusa~\cite{igusa-kroneckerian}.
Letting $Z_0(N)$ denote the plane curve $\Phi_N=0$, we have a map $X_0(N)\to Z_0(N)$ defined on non-cuspidal points by $[(E,C)]\mapsto (j(E),j(E/C))$. Let $k$ be a field with $N\in k^{\times}$ and suppose $(j_1,j_2)$ is a point on $\Phi_N=0$. Then there is an elliptic curve $E_1$ and cyclic subgroup $C$ defined over $\overline{k}$ such that $j(E_1)=j_1$ and $j(E_1/C)=j_2$. If there are $m$ such cyclic subgroups of $E_1$, we say there are $m$ \textbf{non-equivalent (or distinct) isogenies} from $j_1$ to $j_2$. 

\subsection{Some Modular Forms of Interest}
There are particular modular forms we are interested in for this paper, regarded as formal $q$-series: the weights 4 and 6 normalized (constant coefficient one) Eisenstein series $ G_4(q)$ and $G_6(q)$ and the $j$-function given by \[
j(q) := 1728 \frac{G_4(q)^3}{G_4(q)^3 - G_6(q)^2} .
\]
The $j$-function parametrizes $j$-invariants of elliptic curves over $\cc$; in particular for a given $z \in \cc$ there exists a $q_0$ with $j(q_0) = z$ and the elliptic curve over $\cc$ with short Weierstrass equation \[
E: y^2 = x^3 - \frac{G_4(q_0)}{48}x + \frac{G_6(q_0)}{864}
\]
has $j$-invariant $j(q_0)$. We have the following identities between the Eisenstein series $G_4$ and $G_6$ as well as $j$ and its derivative $j'$ \cite[Proposition 7.1]{Schoof}:
\begin{align}\label{eqn:modular_form_identities}
\frac{j'}{j} &= -\frac{G_6}{G_4}, \quad      \frac{j'}{j-1728} = \frac{-G_4^2}{G_6}.
\end{align}

\subsection{Singularities of Plane Curves}

We introduce terminology related to singularities of plane curves that we use later when computing isogenies corresponding to points in $X_0(\ell)$ lying above a singular point on $\Phi_\ell=0$. 
\begin{definition} (Singularities of Plane Curves)
\cite[Section 3.1]{Fulton}   
Let $F(X,Y) \in k[X,Y]$ be a non-constant polynomial. A point $(a,b) \in k^2$ is a {\em multiple point} or {\em singular point} of $F$ if   \[
F(a,b) = \frac{\del F}{\del X} (a,b) = \frac{\del F}{\del Y}(a,b) = 0.
\] 
Suppose $(a,b)$ is a singular point of $F$. Via change of variables define the polynomial \[
H(X, Y) = F(X+a, Y+b).
\]
As $H(0,0) = F(a,b) = 0$ we can write \[
H(X,Y) = H_m(X,Y) + \ldots + H_{\deg(H)}(X,Y),
\]
where each $H_d$ is homogeneous of degree $d$ and $H_m \neq 0$, and define the \textit{multiplicity} of the singularity $(a,b)$ of $F$ as $m$. 
For an affine plane curve $C:F(X,Y)=0$, a point $(a,b)$ on $C$ is a singular point provided $(a, b)$ is a singular point of $F$.
\end{definition}
We consider smooth points as points of multiplicity one (this is consistent with the above definition of multiplicity).
The important property of singular points to remember for later is the following elementary lemma, a generalization of classifying the order of vanishing of holomorphic functions at a point via local power series representation:
\begin{lemma} \label{lem:singular_partials_vanish}
Let $k$ be a field,  $H \in k[X,Y]$ a polynomial such that $H(0,0) = 0$, and write $H = H_m + \ldots + H_{\deg(H)}$ where each $H_d$ is homogeneous of degree $d$ and $H_m \neq 0$ is of minimal degree. Then for all pairs $u, v \in \bb{Z}_{\ge 0}$ such that $u+v < m$ the formal partial derivative $\frac{\del^{u+v} H}{\del X^u \del Y^v}$ vanishes at $(0,0)$. Consequently, if $F \in k[X,Y]$ is any polynomial which has a singular point $(a,b)$ of multiplicity $m$ then for all pairs $u,v \in \bb{Z}_{\ge 0}$ with $u+v<m$ the formal partial derivative $\frac{\del^{u+v}F}{\del X^u \del Y^v}$ vanishes at $(a,b)$.
\end{lemma}

\section{Singular Points on \texorpdfstring{$Z_0(\ell)$}{Z0(l)}} \label{sec:singular}
From here on we abuse notation and denote the plane curve $\Phi_{\ell}(X,Y)=0$ by $Z_0(\ell)$. In this section we observe the multiplicity of a point $(j_1,j_2)\in Z_0(\ell)(k)$ equals the number of cyclic $\ell$-isogenies from an elliptic curve with $j$-invariant $j_1$ to an elliptic curve with $j$-invariant $j_2$ provided neither $j_1$ nor $j_2$ are 0 or 1728. We use this in the following section to compute normalized models for $\ell$-isogenies corresponding to singular points of $Z_0(\ell)$. 

\begin{proposition} \label{prop:multi_isog=multipoint} 
Let $k$ be a field that is either a subfield of $\cc$ or a finite field. Let $\ell$ be a prime that is not divisible by $\Char k$ and let $E_1$ and $E_2$ be elliptic curves over $k$. If $\{ j_1 , j_2 \} \cap \{0, 1728 \} = \emptyset$, then the number of nonequivalent $\ell$-isogenies from $E_1$ to $E_2$ is equal to the multiplicity of $(j_1, j_2)$ as a point on $Z_0(\ell)$. 
\end{proposition}

\begin{proof}
If $(j_1,j_2)$ is smooth then there is a unique $\ell$-isogeny from $j_1$ to $j_2$ and conversely. Suppose now there are $m\geq 2$ distinct $\ell$-isogenies from $j_1$ to $j_2$, so $(j_1,j_2)$ is a singular point on $Z_0(\ell)$. 
If $k$ is a subfield of $\cc$, there can be at most two distinct $\ell$-isogenies $j_1\to j_2$, and it is a consequence of the Kronecker congruence relation~\cite[p248]{Schoof} that $(j_1,j_2)$ is an ordinary double point. Suppose $k$ is  a finite field and view $X_0(\ell)$ and $Z_0(\ell)$ as curves over $k$. The map $\eta: X_0(\ell) \twoheadrightarrow Z_0(\ell)$ given on non-cuspidal points by $[E, C] \mapsto ( j(E), j(E/C) )$ is an immersion away from the cusps of $X_0(\ell)$ \cite[Proposition 2.2]{KK} and consequently by \cite[Corollary 2.2]{KK} all non-cuspidal singular points in $Z_0(\ell)(\overline{k})$ are ordinary multiple points. 

If $E_1$ and $E_2$ have $j$-invariants $j_1$ and $j_2$ then $\Aut(E_u)=\{[\pm 1]\}$ which implies the dual map induces a bijection
\[
\{C<E_1: \#C=\ell,\; E_1/C\simeq E_2\} \leftrightarrow \{C<E_2:\#C=\ell, \; E_2/C\simeq E_1\}. 
\]
Thus if there are $m$ distinct $\ell$-isogenies from $j_1$ to $j_2$ and $j_1, j_2 \neq 0, 1728$ there are also $m$ distinct $\ell$-isogenies from $j_2$ to $j_1$.
Then $\Phi(X, j_2) = (X-j_1)^m f_1(X)$ and $\Phi(j_1, Y) = (Y-j_2)^m f_2(Y)$ where $(X-j_1) \nmid f_1(X)$ and $(Y-j_2) \nmid f_2(Y)$. As $m \ge 2$ this implies that both partial derivatives satisfy 
\[
\frac{d \Phi_{\ell}(X, j_2) }{dX} (j_1)  = \frac{\del \Phi_{\ell}}{\del X}(j_1, j_2) = 0 =\frac{\del \Phi_{\ell}}{\del Y}(j_1, j_2) = \frac{d \Phi_{\ell}(j_1, Y)}{dY} (j_2) 
\]
so $(j_1, j_2)$ is a multiple point of $\Phi_\ell$ and by our earlier paragraph $(j_1, j_2)$ is an ordinary multiple point. 

As $(j_1, j_2)$ is an ordinary multiple point the number of tangent lines to $Z_0(\ell)$ at $(j_1, j_2)$ is equal to the multiplicity of $(j_1, j_2)$ on $Z_0(\ell)$. The curve $X_0(\ell)$ is smooth and $\eta: X_0(\ell) \twoheadrightarrow Z_0(\ell)$ is an immersion at all points $[(E, C)] \in \eta^{-1}(j_1, j_2)$, so the points in the fiber $\eta^{-1}(j_1, j_2)$ correspond to the tangent lines to $Z_0(\ell)$ at $(j_1, j_2)$. There are $m$ points in the fiber $\eta^{-1}(j_1, j_2)$ hence $m$ distinct tangent lines to $Z_0(\ell)$ at $(j_1, j_2)$ hence $(j_1, j_2)$ has multiplicity $m$.
\end{proof}

\begin{remark}
Proposition~\ref{prop:multi_isog=multipoint} is false if we drop the assumption $\{j_1, j_2\} \cap \{0, 1728 \} = \emptyset$. The fiber in $X_0(\ell)$ over a point $(j_1, j_2) \in Z_0(\ell)$ with $j_u = 0, 1728$ can be smaller than one might expect. In particular if there are three $\ell$-isogenies from $j_1 = 0$ to $j_2$ and only one $\ell$-isogeny from $j_2$ to $j_1$ the corresponding point $(j_1, j_2)$ on $\Phi_\ell =0$ is smooth because $\frac{\del \Phi_\ell}{\del X}(j_1, j_2) \neq 0$. 
An explicit example is $\Phi_2(0, 54000) = 0$. The curve $y^2=x^3+1$ has $j$-invariant $0$ and for $0\leq u\leq 2$ the cyclic subgroups $C_u=\langle(-\zeta_3^u,0)\rangle$ of order $2$ satisfy $j(E/C_u)=54000$, and $\Aut(E)$ acts transitively on the set $\{C_0,C_1,C_2\}$. Thus there is only one point $[(E,C_u)]$ on $X_0(2)$ above the smooth point $(0,54000)$ despite there being three pairwise non-equivalent isogenies $E\to E/C_0$. 
\end{remark}

We next record two lemmas on how the multiplicity of a multiple point  dictates vanishing of partial derivatives.

\begin{lemma}\label{lem:partials_vanish}
     If $j_1, j_2 \neq 0, 1728$ are $\ell$-isogenous $j$-invariants in a field $k$ which is a subfield of $\bb{C}$ or a finite field of characteristic at least 5 and $\ell \neq p$ a prime with $m \ge 1$ distinct $\ell$-isogenies from $j_1$ to $j_2$ then for all $0 \le u+v < m$ we have $\frac{\del^{u+v} \Phi_{\ell} }{\del X^u \del Y^v} (j_1, j_2) = 0$.
\end{lemma}
\begin{proof}
     This follows from Lemma~\ref{lem:singular_partials_vanish} and Proposition~\ref{prop:multi_isog=multipoint}.
\end{proof}

\begin{lemma}\label{lem:modulimultiplicity=algebraicmultiplicity}
    Suppose $(j_1,j_2)$ is a multiplicity $m$ point on $\Phi_{\ell}=0$ and that $j_1,j_2\in k - \{0,1728\}$. Then the multiplicity of $j_1$ as a root of $\Phi_\ell(X,j_2)$ is $m$. In particular $\Phi_{X^m}(j_1,j_2)\not=0$. 
\end{lemma}

\begin{proof}
    Fix elliptic curves $E_1$ and $E_2$ over $\overline{k}$ with $j(E_i)=j_i$. Since $j(E_1)\not=0,1728$, by Proposition~\ref{prop:multi_isog=multipoint} there are exactly $m$ distinct cyclic isogenies $E_1\to E_2$ of degree $\ell$, in the sense that their kernels are pairwise distinct. Let $n$ be the multiplicity of $j_1$ as a root of $\Phi_\ell(X,j_2)$. Then $n\geq m$ since $(j_1,j_2)$ has multiplicity $m$ by Lemma~\ref{lem:partials_vanish}. Write $\Phi_{\ell}(X,j_2)=(X-j_1)^nf(X)$ where $f(j_1)\not=0$. Then $\deg f = \ell+1 - n$. On the other hand, since $j_2\not=0,1728$, there are $\ell+1-m$ distinct cyclic $\ell$-isogenies $E_2\to E$ with $j(E)\not=j_1$ and $f(j(E))=0$, so $\deg f = \ell+1-n \geq \ell+1-m$. Thus $n\leq m$ and we conclude $m=n$. 
\end{proof}

\section{Computing Normalized Codomain Models at a Singular Point}\label{sec:normalizedcodomain}
We will derive the relevant equations over $\mathbb{C}$ and then transfer the results to characteristic $p \ge 5$ via Deuring lifting. 

Suppose $E_1/\CC$ is an elliptic curve, $\ell$ is a prime, and $j_2=j(E_2)$ is a root of $\Phi_\ell(j_1,Y)$ for $j_1=j(E_1)$. We wish to compute a normalized isogeny $\phi\colon E_1\to E_2$ of degree $\ell$. 
There is some $q_0$ such that $j_1=j(q_0)$ and $j_2=j(q_0^\ell)$, in which case $A_2=-\ell^4G_4(q_0^\ell)/48$ and $B_2=\ell^6G_6(q_0^\ell)/864$ give the Weierstrass coefficients for an elliptic curve $E_2\colon y^2=x^3+A_2x+B_2$ admitting a normalized $\ell$-isogeny from $E_1$. Rather than computing $\tau_0\in Y_0(\ell)=\Gamma_0(\ell)\backslash\mathbb{H}$ to find $q_0=\exp(2\pi i \tau_0)$ satisfying $j(q_0)=j_1,j(q_0^\ell)=j_2$, it suffices to calculate $j'(q_0^\ell)$ as Equation~\ref{eqn:modular_form_identities} states the codomain of the normalized isogeny has Weierstrass coefficients\footnote{Note that  in \cite[p247-248]{Schoof}, Schoof omits the factors of $\ell$ in $A_2$ and $B_2$; with this choice, the {\em dual} of $\phi$ is normalized.
There is a typo on page 250 where Schoof scales by $\ell^{-4}$, $\ell^{-6}$ rather than $\ell^4$ and $\ell^6$ to obtain the normalized codomain, and there is a sign error for $\widetilde{B}$.
} 
\begin{equation}\label{eqn:normalized_codomain}
A_2=\frac{-\ell^4G_4(q_0^\ell)}{48}=\frac{-\ell^4j'(q_0^\ell)^2}{48j_2(j_2-1728)}, \quad B_2=\frac{\ell^6G_6(q_0^\ell)}{864}=\frac{-\ell^6j'(q_0^\ell)^3}{864j_2^2(j_2-1728)}  .  
\end{equation}

We now explain how to calculate $r=j'(q_0^\ell)$ assuming $j_1,j_2\not\in \{0,1728\}$. Consider the parameterization of $Z_0(\ell):\Phi_\ell=0$ in $\CC$
\begin{align*}
    \gamma\colon \{z\in \CC:|z|<1\}&\to Z_0(\ell) \\ 
    q &\mapsto (j(q),j(q^\ell)). 
\end{align*}
Suppose $P_0=\gamma(q_0)=(j_1,j_2)$ is a point of multiplicity $m$ on $Z_0(\ell)$. Consider the differential operator $q\frac{d}{dq} = \frac{d}{dz}$ where $z(\tau)=2\pi i \tau$ and set $\gamma'=\frac{d}{dz}\gamma$. Let $D_v=v_1\frac{\partial}{\partial x}+v_2\frac{\partial}{\partial y}$ denote the directional derivative in the direction $v=(v_1,v_2)$. Viewing $\Phi_\ell(j(q),j(q^\ell))=0$ as a function of $q$ and taking $m$ derivatives with respect to $z$ yields
\[
0 = \frac{d^m}{dz^m}\Phi_\ell(\gamma)|_{q=q_0} = D^m_{\gamma'(q_0)}(\Phi_\ell(X,Y))|_{(X,Y)=P_0}
\]
because all partial derivatives of $\Phi_\ell(X,Y)$ up to weight $m$ vanish at $P_0$. 
Noting $\gamma'(q_0)=(j_1',\ell r)$ for $j_1'=18j_1B_1/A_1$ and considering the vector $v=(j_1',\ell t)$ for $t$ a formal variable yields $D^m_{(j_1',\ell t)}\Phi_\ell(X,Y)(P_0)$ is a polynomial in $t$ of degree $m$, and among its $m$ roots is the value $r$ we seek.

The modular curve $X_0(\ell)$ is the normalization of the curve $Z_0(\ell):\Phi_\ell=0$. View $v=[v_1:v_2]\in \mathbb{P}^1(\CC)$ as a vector in $\CC^2$ and again fix $P_0$ a point of multiplicity $m$ on $Z_0(\ell)$. Then $D^m_v\Phi_\ell(P)$ is homogeneous of degree $m$ in the variables $v_1$ and $v_2$, and its zeroes correspond to the tangent directions to $Z_0(\ell)$ at $P_0$. These tangent directions correspond to the points on $X_0(\ell)$ above $P_0$, and the $m$ zeroes of $D^m_v\Phi_\ell(P_0)$ in $\mathbb{P}^1(\CC)$ give these directions. Setting $v_1=j_1'=18j_1B_1/A_1$ and $v_2=\ell t$, we again see the zeroes of $D^m_{(j_1',\ell t)}\Phi_\ell(P_0)$ correspond to normalized $\ell$-isogenies from $E_1$ to elliptic curves of $j$-invariant $j_2$.

\subsection{Fiber Polynomials}
In the following section, we will need to compute $\frac{d}{dz}j(q^\ell)|_{q=q_0}$ when $P_0=(j(q_0),j(q_0^\ell))$ is  a multiplicity $m$ point of $Z_0(\ell)$ by finding it as a root of the polynomial $D^m_{(j_1',\ell t)}\Phi_\ell(P_0)$.

\begin{definition}[Fiber Polynomial] \label{def:fiber_polynomial}
    Suppose $j_1,j_2\neq 0,1728\in k$ such that $P=(j_1,j_2)$ is a singular point of multiplicity $m$ on $Z_0(\ell) \colon \Phi_{\ell}(X,Y)=0$, fix a model $E_1: y^2 = x^3 + A_1x + B_1$ for an elliptic curve with $j$-invariant $j_1$, and define $j_1' = 18 \frac{B_1}{A_1} j_1$. For each $n\geq0$ define the polynomial
    \[
    F_{P,n}(t) = D^n_{(j_1',\ell t)}\Phi(P)=\sum_{u=0}^n \binom{n}{u} \ell^{n-u} (j_1')^u  \Phi_{X^u Y^{n-u}}(j_1,j_2)t^{n-u}.
    \]
    The \textit{fiber polynomial} associated to $P$ and $E_1$ is the polynomial
\[
F_P(t) = F_{P,m}(t)=\sum_{u=0}^m \binom{m}{u} \ell^{m-u} j_1'^u \Phi_{X^u Y^{m-u}}(j_1, j_2) t^{m-u} \in k[t].
\]
\end{definition}

Suppose $P=(j_1,j_2)$ is a singular point of $\Phi_\ell$ and $j(E_1)=j_1$. We're interested in computing an $\ell$-isogeny above $P$ with domain $E_1$. We will show in Lemma~\ref{lem:multi_normalized} that if the characteristic $p$ of $k$ is sufficiently large relative to $\ell$ then there is at most one normalized $\ell$-isogeny $\phi\colon E_1\to E_2$, so a normalized $\ell$-isogeny is determined by its codomain. 
To find the codomain, we need to find the value $j'(q_0^\ell)$ where $q_0\in \CC$ satisfies $j(q_0)=j_1$ and $j(q_0^\ell)=j_2$. Because $P=(j_1,j_2)$ is a singularity of multiplicity $m$ we expect there to be $m$ normalized $\ell$-isogenies, and we will show this is indeed the case away from $0, 1728$.

\begin{proposition}\label{prop:isogeniestoroots} 
Suppose 
$k$ is a subfield of $\CC$ or a finite field of characteristic $p \ge 5$. Let $\ell \neq p$ be a prime and let $j_1, j_2 \in k - \{0, 1728 \}$ such that $P=(j_1, j_2)$ is a  point of $\Phi_{\ell} \in k[X, Y]$ of multiplicity $m$. 
Fix elliptic curves $E_1: y^2 = x^3 + A_1x + B_1$ and $E_2: y^2 = x^3 + A_2x + B_2$ defined over $k$ with $j$-invariants $j_1, j_2$. If there exists a normalized $\ell$-isogeny $\phi\colon E_1\to E_2$ defined over $k$, then the value $j_2' = \frac{18B_2}{\ell^2A_2}j_2$ is a root of the fiber polynomial $F_P(t)$.

\end{proposition}

\begin{proof}
If $k\subseteq \CC$ then $\Phi_\ell=0$ has only ordinary double points as singularities; this is  a consequence of the Kronecker relation $\Phi_{\ell}(X,Y)\equiv (X^\ell-Y)(Y^\ell-X)\pmod{\ell}$ and the integrality of $j$-invariants of curves with complex multiplication; see~\cite[Proposition 4.6]{MR879933}, for example. Thus $m=1$ or $2$. From the preceding discussion, we have 
\[
F_P(j_2') = (D_{(j_1',\ell j_2')}^m\Phi_\ell)(P)=0.
\]

For the remainder of the proof we assume $k$ is a finite field of characteristic $p\geq 5$. By the Deuring lifting theorem, there exists a $q_0 \in \CC$, 
a discrete valuation ring $R$ with maximal ideal $\P$, residue field $k$, fraction field $L\subseteq \CC$, and elliptic curves $E_i^*: y^2 = x^3 + A_i^*x + B_i^*$ over $L$ with good reduction at $\P$ such that:
    \begin{enumerate}
        \item $j(q_0), j(q_0^\ell) \in R$ and $j(q_0), j(q_0^\ell) \equiv j_1, j_2  \pmod{\P}$ respectively. 
        \item Each $A_i^*, B_i^* \in R$ and
    \begin{align*}
A_1^*=\frac{- G_4(q_0)}{48} & \equiv A_1\pmod{\f{P}} & B_1^* = \frac{G_6(q_0)}{864} &\equiv B_1 \pmod{\f{P}} \\ 
A_2^*=\frac{-\ell^4 G_4(q_0^\ell)}{48} & \equiv A_2 \pmod{\f{P}} & B_2^*=\frac{\ell^6 G_6(q_0^\ell)}{864} & \equiv B_2 \pmod{\f{P}}.
\end{align*}
\item There is a normalized $\ell$-isogeny $\phi^*: E_1^* \to E_2^*$ such that $ \phi \equiv \phi^* \pmod{\P}$. 
\end{enumerate}

Now choose $q_0\in \CC$ so that $j(q_0)=j_1^*:=j(E_1^*)$ and $j(q_0^\ell)=j_2^*:=j(E_2^*)$. We have 
\[
j_1'^* = j'(q_0)=\frac{18j_1^*B_1^*}{A_1^*},\quad j_2'^*:= 
j'(q_0^\ell)=\frac{18j_2^*B_2^*}{\ell^2A_2^*}\in R
\]
and $\gamma'^*\coloneqq (j_1'^*,\ell j_2'^*)\in \mathbb{A}^2(R)$ is tangent to $\Phi_\ell=0$ at $(j_1^*,j_2^*)\in \mathbb{A}^2(R)$; we claim the reduction $(j_1',\ell j_2')$ is a tangent direction to $\Phi_\ell=0$ at $(j_1,j_2)$ as well. The isogeny $\phi$ corresponds to a point on $X_0(\ell)$, the normalization of $Z_0(\ell)$, above $P$. Write 
\[
\Phi_\ell(j_1+X,j_2+Y)=H_m(X,Y)+H_{m+1}(X,Y)+\cdots,
\]
where $H_i$ is homogeneous of degree $i$; this is possible because we assumed $P$ is a point of multiplicity $m$ on $Z_0(\ell)$. Fixing a local parameter $u$ for this branch, so $X(u)=j_1'u+O(u^2)$ and $Y(u)=\ell j_2'u+O(u^2)$, we have 
\[
0 = \Phi_\ell(j_1+X(u),j_2+Y(u))=H_m(X(u),Y(u))+O(u^{m+1})=u^mH_m(j_1',\ell j_2')+O(u^{m+1})
\]
and conclude $H_m(j_1',\ell j_2')$ must be $0$. On the other hand, we have
\[
F_P(j_2')=D^m_{{(j_1',\ell j_2')}}\Phi_\ell(P) = m!H_m(j_1',\ell j_2')=0
\]
 as desired.

\end{proof}

\begin{remark}
    Suppose $\ell$ is a prime and $k$ is a field of characteristic $0$ or $p>4\ell$, that $E_1$ is an elliptic curve over $k$, that $j_2\in k$, that $P=(j_1,j_2)$ is a (possibly singular) point on $Z_0(\ell)$, and neither $j_1=j(E_1)$ nor $j_2$ equals either $0$ or $1728$. If there is an extension $K/k$ such that $E_1$ admits a normalized $\ell$-isogeny to some curve $E_2$ over $K$ with $j$-invariant $j_2$, then Proposition~\ref{prop:isogeniestoroots} implies the corresponding root of the fiber polynomial $F_P(t)$ is an element of $K$.
\end{remark}

Proposition \ref{prop:isogeniestoroots} says that given a fixed singular point $P=(j_1, j_2)$ of multiplicity $m$ on $Z_0(\ell) : \Phi_\ell(X, Y) = 0$ over a finite field $k$ of sufficiently large characteristic we can compute a polynomial $F_P(t)$ whose roots in $\overline{k}$ correspond to elliptic curves $E_2$ such that $E_1$ admits a normalized $\ell$-isogeny $\phi\colon E_1\to E_2$.
There are $m$ distinct $\ell$-isogenies from $j_1$ to $j_2$ which encode $m$ roots of $F_P$, but these roots need not be distinct.

\begin{example} \label{ex:fiber_poly_inseparable}
     Consider the curves $E_1:y^2=x^3+x+4$ and $E_2:y^2=x^3+12x+7$ over $\bb{F}_{169}$. These curves are isomorphic, supersingular, and are connected by two normalized $\bb{F}_{169}$-rational $5$-isogenies with distinct kernels. The kernel polynomials are the two degree $2$ factors of $x^4 + 10x^3 + 9x^2 + 12x + 6$ over $\bb{F}_{169}$. 
\end{example}

This pathology in the example above cannot happen in characteristic $0$ or if the characteristic is sufficiently large relative to $\ell$:
\begin{lemma}\label{lem:multi_normalized}
 
    Let $E_1,E_2$ be elliptic curves over a field $k$. Suppose there exist two distinct normalized, separable $\ell$-isogenies $\phi,\psi\colon E_1\to E_2$. Then $0<\Char(k)<4\ell$. 
\end{lemma}
\begin{proof}
Suppose $\phi,\psi\colon E_1\to E_2$ are two distinct, normalized isogenies of the same degree. Then $\phi-\psi$ is an inseparable, nonzero isogeny, so $p=\Char(k)\not=0$ and $p\mid \deg (\phi-\psi)$. Consider the binary quadratic form $\deg(m\phi-n\psi)$. We have
\[
\deg(m\phi-n\psi)=\ell m^2 -mnt+\ell n^2
\]
where $t$ is the trace of the endomorphism $\phi\circ\widehat{\psi}$. The discriminant of this form is $t^2-4\ell^2\leq 0$. Since $\phi\not=\psi$ we have $\deg(\phi-\psi)=2\ell -t>0$. Thus $0<2\ell - t \leq 4\ell$, and $p\mid (2\ell-t)$, so $p \le 4\ell$ and the claim follows. 

\end{proof}
By Lemma~\ref{lem:multi_normalized} if $\Char k =0$ or $\Char k > 4\ell$ then a normalized $\ell$-isogeny is uniquely determined by its domain and codomain, and in the affirmative case our main theorem below shows the normalized $\ell$-isogenies of $E_1$ to curves of $j$-invariant $j_2$ are in bijection with the roots of the fiber polynomial $F_P(t)$.

\begin{theorem}\label{thm:isogeny_FP_root_bijection}
    Assume either $\Char(k)=0$ or  $\Char(k)=p > 4 \ell$.  Suppose $j_1,j_2\in k$ are neither $0$ nor $1728$. Fix a $k$-rational point $P = (j_1, j_2)$ of multiplicity $m$ on $Z_0(\ell)$, an elliptic curve $E_1$ over $k$ with $j$-invariant $j_1$, and the associated fiber polynomial $F_P(t)$. Then there is a bijection between the roots of $F_P(t)$ in $k$ and the cyclic normalized $\ell$-isogenies of $E_1$ to curves with $j$-invariant $j_2$ defined over $k$ given by the following: 
    \begin{enumerate}
    \item Given an elliptic curve $E_2\colon y^2=x^3+A_2x+B_2$ with $j$-invariant $j_2$ admitting a normalized cyclic separable $\ell$-isogeny from $E_1$, the quantity $r = \frac{18 B_2}{\ell^2 A_2} j_2$ is a root of $F_P(t)$ in $k$.
\item Given a root $r\in k$ of $F_P(t)$, define
\[
   E_r\colon y^2 =x^3+\frac{- \ell^4r^2}{48j_2(j_2-1728)} x+ \frac{-\ell^6r^3}{864j_2^2(j_2-1728)}.
\]
Then there is a normalized $\ell$-isogeny $E_1\to E_r$ defined over $k$. 
\end{enumerate}
\end{theorem}

\begin{proof}

    As $P$ is a point of multiplicity $m$ on $Z_0(\ell)/k$ satisfying the hypotheses of Proposition \ref{prop:multi_isog=multipoint}, there exist $m$ different $\ell$-isogenies of $E_1$ with distinct kernels to curves with $j$-invariant $j_2$. The codomain of each isogeny yields a root $r$ of $F_P(t)$ by Proposition~\ref{prop:isogeniestoroots} and by Lemma~\ref{lem:multi_normalized} these curves are all distinct. Next we show the corresponding roots are distinct.
    
    Let $E_2\colon y^2=x^3+A_2x+B_2$ and $E_2'\colon y^2=x^3+A_2'x+B_2'$ be two  curves with $j$-invariant $j_2$ such that there are normalized $\ell$-isogenies from $E_1$ to $E_2$ and from $E_1$ to $E_2'$.  Since $j_2=j(E_2)=j(E_2')$, the curves $E_2$ and $E_2'$ are isomorphic: there exists $u\in \overline{k}$ such that $A_2'=u^4A_2$ and $B_2'=u^6B_2$.  
  Observe that if 
    \[
    r = {j_2}\frac{18B_2}{\ell^2A_2} = {j_2}\frac{18B_2'}{\ell^2A_2'} = r'
    \]
    then 
    \[
    \frac{B_2}{A_2} = \frac{B_2'}{A_2'}=u^2\frac{B_2}{A_2}.
    \]
    Therefore $u=\pm1$ which implies $(A_2,B_2)=(A_2',B_2')$ and consequently $E_2 = E_2'$. Thus the map from models to roots of $F_P(t)$ is an injection. As $F_P(t)$ has degree $m$, the map is surjective. In characteristic $0$, the formula to obtain the Weierstrass coefficients for $E_r$ from $r$ is given by Equation~\ref{eqn:normalized_codomain}, which follows from the power series identities in \cite[Proposition 7.1]{Schoof} applied to $j(q_0^\ell), j'(q_0^\ell), G_4(q_0^\ell), G_6(q_0^\ell)$ to manipulate the expressions for $r^2$ and $r^3$. The result in positive characteristic follows from the Deuring Lifting Theorem. 

    Finally we argue the claim about the field of definition of the normalized isogeny. It is clear that a normalized $\ell$-isogeny $\phi\colon E_1\to E_2$ defined over $k$ yields a root in $k$. Conversely, if $r\in k$ is a root of $F_P(t)$ then there is some isogeny $\phi\colon E_1\to E_r$ normalized of degree $\ell$, defined over a finite extension $K$ of $k$. We claim that $K=k$. Let  $\sigma\in \Gal(\overline{k}/k)$ be arbitrary. Then  $\phi^\sigma\colon E_1\to E_r$ is also normalized of degree $\ell$ so Lemma~\ref{lem:multi_normalized} implies $\phi=\phi^\sigma$. Thus $\phi$ is defined over $k$.

\end{proof}
\begin{remark}
    Suppose $k$ is a field of characteristic $p$, that $\phi\colon E_1\to E_2$ is a normalized $\ell$-isogeny defined over $k$,  and $j(E_1),j(E_2)\in k$ give a point $P=(j(E_1),j(E_2))$  of multiplicity $m$ on $Z_0(\ell)$. Defining $j_1'=18j_1B_1/A_1$ as before, the fiber polynomial $F_P(t)$ satisfies 
    \[
    F_P(t) = D^m_{(j_1',\ell t)}\Phi_\ell(P) = m!H_m(j_1',\ell t)
    \]
    where $\Phi_\ell(j_1+X,j_2+Y)=H_m(X,Y)+H_{m+1}(X,Y)+\cdots$ as before. 
    If $p \leq m$ then $F_P(t)$ is identically zero. If $p>m$ then Proposition~\ref{prop:isogeniestoroots} implies that the normalized isogeny $\phi$ induces a root of $F_P(t)$ in $k$. Conversely, if $p>4\ell >m$, the association of normalized isogenies to roots of $F_P(t)$ is a bijection by Theorem~\ref{thm:isogeny_FP_root_bijection}.
\end{remark}

\section{Computing Kernel Polynomials at a Singular Point}\label{sec:abscissasum}
For an isogeny $\phi: E_1 \ra E_2$, we define its {\bf kernel polynomial} $f_{\phi}$ of $\phi$ to be
the monic polynomial vanishing precisely at the $x$-coordinates of the nonzero points in $\ker\phi$. Define also 
\[
\sigma_{\phi} = \sum_{0\neq P \in \ker\phi} x_P.
\]
If $\deg\phi$ is odd, then $f_{\phi}(t)=t^d -\frac{\sigma_\phi}{2} t^{d-1} + \ldots$.

 If $E_1$ and $E_2$ are connected by a normalized $\ell$-isogeny over a field $k$ of characteristic 0 or characteristic $p > 8 \ell - 5$, then  the algorithm {\rm \sf fastElkies'} \cite[Section 4.3]{BMSS} of Bostan--Morain--Salvy--Schost can be used to recover the kernel polynomial $f_\phi$ and subsequently the rational map corresponding to the isogeny in $O(\M(\ell) \log(\ell))$ field operations ~\cite[Theorem 2 (2)]{BMSS}.

 If $\Char(k) = 0$ or $p > 2 \ell - 1$ and the coefficient $\sigma_\phi$ of the kernel polynomial is known, the algorithm {\rm \sf fastElkies}~\cite[Section 4.3]{BMSS} computes $f_\phi$  with only $O(\M(\ell))$ field operations~\cite[Theorem 2]{BMSS}. We explain in this section how to compute the coefficient $-\sigma_\phi/2$ of a kernel polynomial $f_\phi$ for an isogeny $\phi$ above a singular point of $\Phi_\ell=0$ when $\deg\phi=\ell$ is an odd prime.

First we consider elliptic curves over $\cc$ and then appeal to the Deuring Lifting Theorem for finite fields. Suppose $E_1$ is given by $y^2=x^3-\frac{G_4(q_0)}{48}x+\frac{G_6(q_0)}{864}$ and that $\phi\colon E_1\to E_2$ is a normalized $\ell$-isogeny. By \cite[Proposition 7.2.ii]{Schoof}, we have
\[
\sigma_\phi =  \frac{\ell}{12} \llp G_2(q_0) - \ell G_2 \llp q_0^\ell \rrp \rrp.
\]
Application of \cite[Proposition 7.1.iii]{Schoof} yields the equality of $q$-series for the logarithmic derivative 
\[
\frac{j''}{j'} = \frac{1}{6} G_2 - \frac{1}{2} \frac{G_4^2}{G_6} - \frac{2}{3} \frac{G_6}{G_4}.
\]
The above identity gives an equality of $q$-series

\begin{equation}\label{eqn:DrFormula}
\frac{j''(q)}{j'(q)} 
- \ell \frac{j''(q^\ell)}{j'(q^\ell)} 
= \frac{1}{6}\left( G_2(q) -\ell G_2(q^\ell)\right)- 
\frac{1}{2}\left(\frac{G_4(q)^2}{G_6(q)}-\ell\frac{G_4(q^\ell)^2}{G_6(q^\ell)}\right)
-\frac{2}{3}\left(\frac{G_6(q)}{G_4(q)}-\ell \frac{G_6(q^\ell)}{G_4(q^\ell)}\right).
\end{equation}
 Thus, given $G_4(q_0), G_6(q_0), G_4(q_0^\ell), G_6(q_0^\ell)$ and 

 \(\frac{j''(q_0)}{j'(q_0)}-\ell\frac{j''(q_0^\ell)}{j'(q_0^\ell)}\),
 one can compute $\sigma_\phi$; recall that $G_4(q_0^\ell)$ and $G_6(q_0^\ell)$ can be computed at a singular point with the method of the prior section. Even if $G_4(q_0)$ and $G_6(q_0)$ are algebraic, the quantity $G_2(q_0)$ need not be. But if $q_0$ yields an $\ell$-isogeny of  elliptic curves $\phi \colon j(q_0)\to j(q_0^\ell)$ defined over a number field, the quantity $G_2(q_0)-\ell G_2(q_0^\ell)$ will be algebraic since it is a multiple of the coefficient $-\sigma_\phi/2$ of the kernel polynomial of the $\ell$-isogeny $\phi$. Thus rather than trying to find expressions for $G_2(q_0)$ and $G_2(q_0^\ell)$ one  instead computes \(\frac{j''(q)}{j'(q)}-\ell\frac{j''(q^\ell)}{j'(q^\ell)}\). We proceed to compute this quantity when $(j_1,j_2)$ is a possibly singular point. 

At a smooth point computing the second derivative $\Phi(j(q), j(q^\ell))''$ allows one to solve an equation for $\frac{j''(q)}{j'(q)} - \ell \frac{j''(q^\ell)}{j'(q^\ell)}$ provided at least one of $\Phi_{XX}, \Phi_{XY}, \Phi_{YY}$ is nonzero at $(j(q_0), j(q_0^\ell))$. 
 If $(j_1,j_2)$ is a singularity of multiplicity $m$ we will need to compute $m+1$  derivatives of $\Phi$ in order to determine an equation for $\frac{j''}{j'}(q_0) - \ell \frac{j''}{j'} (q_0^\ell)$ at a singular point. The following can be proven by induction:
\begin{lemma} \label{lem:diffeq2} 
Fix an integer $m \ge 2$, a prime $\ell$, and define $\Phi = \Phi_\ell(j(q),j(q^\ell))$. Then 
\[
\frac{d^{m+1} \Phi}{d (2i \pi \tau)^{m+1}} = \e_{m+1}(q) + \mathcal{F}_{m+1}(q) + \binom{m+1}{m-1} \llp \frac{j''(q)}{j'(q)} \mc{H}_{m+1}(q) + \ell \frac{j''(q^\ell)}{j'(q^\ell)} \widetilde{\mc{H}}_{m+1}(q) \rrp,
\]
where for $n \ge 2$ we have 
\[
\mathcal{F}_{n}(q) = \sum_{u=0}^{n} \binom{n}{u} \ell^{n-u} j'(q)^u j'(q^\ell)^{n-u} \Phi_{X^u Y^{n-u}}(j(q),j(q^\ell)) ,
\]
\[
\mc{H}_{n}(q)= \sum_{u=1}^{n-1} \binom{n-2}{u-1} \ell^{n-1-u} j'(q)^{u} j'(q^\ell)^{n-1-u} \Phi_{X^uY^{n-1-u}}(j(q),j(q^\ell)) ,
\]
\[
\widetilde{\mc{H}}_{n}(q)=\sum_{u=0}^{n-2} \binom{n-2}{u} \ell^{n-1-u} j'(q)^u j'(q^\ell)^{n-1-u} \Phi_{X^uY^{n-1-u}}(j(q),j(q^\ell)) ,
\]
and $\e_{m+1} = \sum_{0<u+v<m} \e_{u,v}(q) \Phi_{X^uY^v}(j(q),j(q^\ell))$ where for $0<u+v<m$,  
$\e_{u,v}(q)$ lies in the ring $\zz \left[ j(q), \ldots, j^{(m+1)}(q), j'(q^\ell),\ldots,j^{(m+1)}(q^\ell)\right]$.

\end{lemma}

As before we replace $j'(q^\ell)$ with an indeterminate $t$ in $\mathcal{F}_n, \mathcal{H}_n$, and $\widetilde{\mathcal{H}_n}$. We will use these polynomials to obtain an equation over our field $k$ we can solve for the values of $\sigma_{\phi_i}$.

\begin{lemma} \label{lem:rootrelation}
Let $k$ be a number field or finite field of characteristic $p > \max\{5,m\}$, let $m$ be a positive integer and let $\ell$ be a prime different from the characteristic of $k$. Suppose that $j_1, j_2 \in k - \{0, 1728 \}$ are $\ell$-isogenous $j$-invariants such that $(j_1, j_2)$ is a multiplicity $m$ point of $\Phi_{\ell}(X,Y)$. Let  $E \colon y^2=x^3+Ax+B$ be an elliptic curve over $k$ with $j(E)=j_1$. Define $j_1' = 18 \frac{B}{A} j_1$ and the following elements of $k[t]$ 
\begin{align*}
F_{P}(t) & = \sum_{u=0}^{m} \binom{m}{u} \ell^{m-u} j_1'^u \Phi_{X^uY^{m-u}}(j_1, j_2) \cdot t^{m-u}  \\
H(t) & = \sum_{u=1}^{m} \binom{m-1}{u-1} \ell^{m-u} j_1'^u \Phi_{X^uY^{m-u}} \cdot t^{m-u} \\
\wt{H}(t) & = \sum_{u=0}^{m-1} \binom{m-1}{u}  \ell^{m-u} j_1'^u \Phi_{X^uY^{m-u}} \cdot t^{m-u} .
\end{align*}
Then for every root $r$ of $F_P$ we have 
\[
\frac{dF_P}{dt}(r) = \frac{m}{r} \widetilde{H}(r) = - \frac{m}{r} H(r).
\]
\end{lemma}

\begin{proof}

We can directly compute 
\[
 \frac{dF_P}{dt}  (r) = \sum_{u=0}^{m-1} (m-u) \binom{m}{u} \ell^{m-u} j_1'^u r^{m-1-u} \Phi_{X^uY^{m-u}}(j_1,j_2)
\]
and we apply the binomial identity $(m-u)\binom{m}{u} = m \binom{m-1}{u}$ to deduce 
\[
\frac{dF_P}{dt}(r) = m \sum_{u=0}^{m-1} \binom{m-1}{u} \ell^{m-u} j_1'^u r^{m-1-u} \Phi_{X^uY^{m-u}}(j_1,j_2).
\]
By Lemma \ref{lem:modulimultiplicity=algebraicmultiplicity} we can directly compute $\Phi_{\ell}(X, j_2) = (X-j_1)^m f(X)$ such that $f(j_1) \neq 0$ and thus $\Phi_{X^m}(j_1, j_2) = m! f(j_1) \neq 0$ as we assumed $p > m$, so $F_P(t)$ has a nonzero constant term and thus $r \neq 0$. Then we can factor
\[
\frac{dF}{dt}(r) = \frac{m}{r}  \sum_{u=0}^{m-1} \binom{m-1}{u} \ell^{m-u} j_1'^u r^{m-u} \Phi_{X^uY^{m-u}}(j_1,j_2) = \frac{m}{r} \widetilde{H}(r).
\]
As $F_P(r) = 0$ we can rearrange to get the identity 
\[
\ell^m r^m \Phi_{Y^m}(j_1, j_2) = - \sum_{u=1}^{m} \binom{m}{u} j_1'^{u}  \ell^{m-u} r^{m-u} \Phi_{X^uY^{m-u}}(j_1, j_2)
\]
which lets us compute
\begin{align*}
\widetilde{H}(r) & = \ell^m r^m \Phi_{Y^m}(j_1,j_2) + \sum_{u=1}^{m-1} \binom{m-1}{u} \ell^{m-u}j_1'^u r^{m-u} \Phi_{X^uY^{m-u}}(j_1,j_2) \\
& = \sum_{u=1}^{m-1} \binom{m-1}{u} \ell^{m-u}j_1'^u r^{m-u} \Phi_{X^uY^{m-u}}(j_1,j_2) \\
& \quad \quad \quad - \sum_{u=1}^{m} \binom{m}{u} \ell^{m-u} j_1'^u r^{m-u} \Phi_{X^uY^{m-u}}(j_1,j_2).
\end{align*}
Finally, by recognizing $\binom{m-1}{u} + \binom{m-1}{u-1} = \binom{m}{u}$ we can make the conclusion
\[
\widetilde{H}(r) = - \sum_{u=1}^m \binom{m-1}{u-1} \ell^{m-u} j_1'^u r^{m-u} \Phi_{X^uY^{m-u}}(j_1,j_2) = - H(r).
\]
\end{proof}

Now we can prove the main theorem of this section. 

\begin{theorem} \label{thm:main}
Let $\ell$ be an odd prime, let $m$ be a positive integer, and let $k$ be a subfield of $\bb{C}$ or a finite field of characteristic $p >4\ell$.   Let $P=(j_1, j_2) \in k^2$ be a  multiplicity $m$ point of $\Phi_{\ell}(X,Y)=0$ with $j_1, j_2 \neq 0, 1728$. Let $E\colon y^2=x^3+Ax+B$ have $j$-invariant $j_1$ with $A,B\in k$ and define $j_1' = 18 j_1\frac{B}{A}$. Define the polynomials
\begin{align*}
F_P(t) =\sum_{u=0}^m \binom{m}{u} \ell^{m-u} j_1'^u \Phi_{X^uY^{m-u}}(j_1, j_2) t^{m-u} ,\; \quad \quad   \\
F_{m+1}(t) = \sum_{u=0}^{m+1} \binom{m+1}{u} \ell^{m+1-u}j_1'^u \Phi_{X^uY^{m+1-u}}(j_1, j_2) t^{m+1-u}.
\end{align*}
For each root $r$ of $F_P(t)$, define the elliptic curve
\[
   E_r\colon y^2 =x^3+\frac{- \ell^4r^2}{48j_2(j_2-1728)} x+ \frac{-\ell^6r^3}{864j_2^2(j_2-1728)}.
\]
and also
\[
\delta_r = \frac{2}{r (m+1)} \frac{F_{m+1}(r)}{\frac{dF_P}{dt}(r)} .
\]
Then there is a normalized $\ell$-isogeny $\phi_r: E \ra E_r$ such that
\[
\sigma_{\phi_r} = \ell  \llp \frac{\delta_r}{2} + \frac{1}{4} \llp \frac{j_1'}{1728-j_1}  - \ell \frac{r}{1728-j_2} \rrp  +  \frac{1}{3} \llp \ell \frac{r}{j_2} - \frac{j_1'}{j_1} \rrp   \rrp.
\]
\end{theorem}

\begin{proof}
We will prove the theorem in the case that $k$ is a finite field of characteristic $p\geq 5$. By Theorem~\ref{thm:isogeny_FP_root_bijection}, for each root $r$ of $F_P(t)$ there is an elliptic curve \[
E_r \colon y^2 =x^3+\frac{- \ell^4r^2}{48j_2(j_2-1728)} x+ \frac{-\ell^6r^3}{864j_2^2(j_2-1728)}
\]
and a normalized $\ell$-isogeny $\phi_r \colon E \to E_r$.

Fix a root $r$ of $F_P(t)$ and the corresponding elliptic curve $E_r : y^2 = x^3 + A_r x + B_r$. Then we can use the Deuring Lifting Theorem to find an order $\mc{O}$ in a number field $L$, a prime ideal $\f{P} \sus \mc{O}$, and a $q_0 \in \bb{C}$ such that the following hold
\begin{itemize}
    \item $k\subseteq \mc{O}/\f{P}$ and $j(q_0) \equiv j_1\pmod{\f{P}}$, $j'(q_0) \equiv j_1'\pmod{\f{P}}$, $j(q_0^\ell) \equiv j_2\pmod{\f{P}}$, and $j'(q_0^\ell) \equiv r\pmod{\f{P}}$.
    \item $G_4(q_0) \equiv - 48 A \pmod{\f{P}}$, $G_6(q_0) \equiv 864B \pmod{\f{P}}$, $G_4(q_0^\ell) \equiv - \frac{48}{\ell^4}A_r \pmod{\f{P}}$, and $ G_6(q_0^\ell) \equiv \frac{864}{\ell^6} B_r\pmod{\f{P}}$. 
    \item There are elliptic curves $E^*$ and $ E_r^*$ defined over $L$ such that $\f{P}$ is a prime of good reduction for both curves and $\llp E^* \rrp_{\f{P}} = E$ and  $\llp E_r^* \rrp_{\f{P}} = E_r$. 
\end{itemize}

Since $\Phi(j(q),j(q^\ell))$ is identically $0$ as a $q$-series, so too is $\llp \Phi(j(q), j(q^\ell)) \rrp^{(m+1)}$. 

We would like to simply 
evaluate the equation in Lemma~\ref{lem:diffeq2} at $q=q_0$ and reduce modulo $\mathfrak{P}$, 

but the second-order and higher derivatives of $j$ may not be integral (or even algebraic) at $q_0$. Instead, we work directly with $q$-series. 

As $q$-series the $j$-function and its $2i\pi \tau$-derivatives are of the form 
\[
j^{(n)} = \frac{(-1)^n}{q} + J^{(n)}(q)
\]
where $J^{(n)}(q) \in \ZZ[[q]]$ and similarly
\[
(j(q^\ell))^{(n)}=\ell^nj^{(n)}(q^\ell)=\frac{(-\ell)^n}{q^\ell} + \ell^n J^{(n)}(q^\ell).
\]
Because the coefficients of both $j^{(n)}(q), j^{(n)}(q^\ell)$ all lie in $\ZZ$ they are both elements of $\OO_{\P}((q))$; similarly for the functions $\mathcal{F}_m(q), \e_m(q),$ and $ \Phi_{X^u Y^v}(j(q), j(q^\ell))$. 
In terms of $A_2, B_2$ we have 
$
j'(q_0^\ell)=18 \frac{B_2^*}{\ell^2 A_2^*} j(q_0^\ell) \equiv 18 j_2\frac{B_2}{\ell^2 A_2}\pmod{\P}$. Finally define 
\[
f(q)\coloneqq \frac{j''(q)}{j'(q)}-\ell \frac{j''(q^\ell)}{j'(q^\ell)} \in \OO_{\f{P}}((q))
\]

Consider the ideal
\[I=\P \OO_\P((q))+(j(q) - j(q_0), j'(q) - j'(q_0), j(q^\ell)-j(q_0^\ell), j'(q^\ell)-j'(q_0^\ell),f(q)-f(q_0))
\]
of $\OO_\P((q))$ and the $k$-algebra $S=\OO_\P((q))/I$. 

As $\Phi_\ell(j(q), j(q^\ell))$ is identically zero as an element of $\ZZ((q))$ so are each of its $2 i \pi \tau$-derivatives, therefore Lemma \ref{lem:diffeq2} gives an equality in $S$ 
\[
-\e_{m+1}(q) \equiv \mathcal{F}_{m+1}(q)+\binom{m+1}{m-1} \llp \frac{j''(q)}{j'(q)} \mc{H}_{m+1}(q) + \ell \frac{j''(q^\ell)}{j'(q^\ell)} \widetilde{\mc{H}}_{m+1}(q) \rrp \pmod{I}.
\]
Examining $\e_{m+1}$ mod $I$ yields an equality in $S$
\[
\e_{m+1}(q) \equiv \sum_{0<u+v<m} \e_{u,v}(q) \Phi_{X^u Y^v}(j_1, j_2) \pmod{I}.
\]
By Lemma~\ref{lem:partials_vanish} each $\Phi_{X^u Y^v}(j_1, j_2)$ is zero in $k$ hence is zero in $S$. We conclude $\e_{m+1} \equiv 0 \pmod{I}$ hence 
\[
\mathcal{F}_{m+1}(q)+\binom{m+1}{m-1} \llp \frac{j''(q)}{j'(q)} \mc{H}_{m+1}(q) + \ell \frac{j''(q^\ell)}{j'(q^\ell)} \widetilde{\mc{H}}_{m+1}(q) \rrp \equiv 0 \pmod{I}.
\]
Let $H(t), \wt{H}(t)$ be the polynomials defined in Lemma $\ref{lem:rootrelation}$. We have the congruences $ H(r) \equiv \mc{H}_{m+1}(q_0)\pmod{\f{P}}$ and $ \wt{H}(r) \equiv \wt{\mc{H}}_{m+1}(q_0)\pmod{\f{P}} $. Since $p>4\ell$, Theorem~\ref{thm:isogeny_FP_root_bijection} guarantees there are $m$ distinct roots of the degree $m$ polynomial $F_P(t)$ so  $\frac{dF_P}{dt}(r)$ is nonzero and each root $r$ corresponds to a  normalized $\ell$-isogeny $E\to E_r$. The results of Lemma \ref{lem:rootrelation} give $H(r) = - \frac{r}{m} \frac{dF_P}{dt}(r)$ and $\wt{H}(r) = \frac{r}{m} \frac{dF_P}{dt}(r)$. 

Returning to our equation and substituting in the values prior mentioned,
\begin{align*}
0 & \equiv  F_{m+1}(r) + \binom{m+1}{m-1} \llp \frac{j''(q)}{j'(q)}  H(r) + \ell \frac{j''(q^\ell)}{j'(q^\ell)} \wt{H}(r)     \rrp \pmod{I} \\
& \equiv F_{m+1}(r)  -\frac{r}{m} \frac{dF_P}{dt}(r)\binom{m+1}{m-1} \cdot \llp \frac{j''(q)}{j'(q)}   - \ell \frac{j''(q^\ell)}{j'(q^\ell)} \rrp \pmod{I}.
\end{align*}
As mentioned above, we have $\frac{j''(q_0)}{j'(q_0)}   - \ell \frac{j''(q_0^\ell)}{j'(q_0^\ell)}\in \OO_{\f{P}}$. Define $s\in k$ to be the image of $\frac{j''(q_0)}{j'(q_0)}   - \ell \frac{j''(q_0^\ell)}{j'(q_0^\ell)}$. We now have the following congruence:
\[
0\equiv F_{m+1}(r)  -\frac{r}{m} \frac{dF_P}{dt}(r)\binom{m+1}{m-1} \cdot s \pmod{I}
\]
The right-hand side is an element of $k\subseteq S$ so 
\[
s = \frac{m}{r} \frac{F_{m+1}(r) }{\binom{m+1}{m-1}\frac{dF_P}{dt}(r)} = \delta_r.
\]
The claimed formula for $\sigma_\phi$ now follows from Equation \ref{eqn:DrFormula}.

\end{proof}

\section{Algorithms and Complexity}
\label{sec:complexity}

In this section, we discuss the complexity of an algorithm for computing all $\ell$-isogenies out of an elliptic curve $E$ over $k$ under the assumption $j_1=j(E)$ is neither $0$ nor $1728$ and $E$ is not $\ell$-isogenous to an elliptic curve of $j$-invariant $0$ or $1728$. The algorithm proceeds by computing $\Phi_\ell(j_1,Y)$, finding its roots in $k$, and then computing the codomains of the $m$ normalized $\ell$-isogenies from $E$ to curves of $j$-invariant $j_2$ in $k$ for each root $j_2$ using the formulas in Section~\ref{sec:normalizedcodomain}, where $m$ is the multiplicity of the point $(j_1,j_2)$ on $\Phi_\ell=0$. We give the results in terms of the cost $\M(n)$ of multiplying two $n$-bit integers.

\begin{lemma}\label{lem:multipoint_complexity}
    Let $k=\FF_{p^e}$ be a finite field with $\log e = O(\log p)$ and set $n=e \log p$. Suppose $j_1,j_2\not=0,1728$ and that $P=(j_1,j_2)$ is a $k$-rational point of multiplicity $m$ on $\Phi_\ell(X,Y)=0$ and assume $\Phi_\ell(X,Y)$ has been precomputed. Then the $k$-rational normalized cyclic $\ell$-isogenies above $P$ can be computed in $
    O(m\ell^2\M(n)
    +mn\M(mn) )$ bit operations. 

\end{lemma}

\begin{proof}
     We will compute the $m$ isogenies above $P$ via the {\rm\sf fastElkies} algorithm of Bostan--Morain--Salvy--Schost~\cite{BMSS}. 

     First, we compute the fiber polynomial $F_P(t)$.
     Let $D_P=j_1'\partial_X+\ell t\partial_Y$ be the directional derivative operator in the direction $(j_1',\ell t)$ (with $t$ a formal variable). Then $F_P(t) = D^m\Phi(P)$ and $F_{m+1}(t)=D^{m+1}\Phi(P)$.
The polynomial $F_P(t)$ is separable of degree $m$ and its zeroes give us possible values of $j_2'$.  We wish to compute it. 

Expanding $f(s,t)=\Phi_{\ell}(j_1+j_1's,j_2+\ell t s)$ as a power series in $s$, let the coefficient of $s^i$ be denoted $f_i$, so $F_P(t) = m!f_m$. We compute $F_P(t)$ by computing $f_m$. This can be done as follows: define $\alpha_{d,u}$ and $\beta_{e,v}$ via 
\[
(j_1+j_1's)^u = \sum_{d=0}^u \alpha_{d,u}s^d, \quad (j_2+\ell t s)^v = \sum_{e=0}^v\beta_{e,v}s^e.
\]
We have $\alpha_{0,1}=j_1,\alpha_{1,1}=j_1'$ and $\beta_{0,1}=j_2, \beta_{1,1}=\ell t$, and the formulas 
\[
\alpha_{d,u+1}=j_1\alpha_{d,u}+j_1'\alpha_{d-1,u}, \quad \beta_{e,v+1}=j_2\beta_{e,v}+\ell t \beta_{e-1,v}. 
\]
Using these initial values and recursive formulas we can compute all $\alpha_{d,u},\beta_{e,v}$ for $d\leq u \leq \ell+1$ and $e\leq v \leq \ell+1$ with $O(\ell^2)$ field multiplications. Now let $c_{u,v}$ be the coefficient of $X^uY^v$ in $\Phi_\ell(X,Y)$. We have
\[
F_P(t) = m!\sum_{u,v} c_{u,v}\sum_{i=0}^m \alpha_{m-i,u}\beta_{i,v}
\]
which we can compute by initializing $f=0$, scanning the coefficients $c_{u,v}$, computing $f_{u,v}=c_{u,v}\sum_{i=0}^m \alpha_{m-i,u}\beta_{i,v}$ costing $O(m)$ multiplications to update $f=f+f_{u,v}$. This costs in total $O(\ell^2m)$ multiplications to compute $F_P(t)$. We can similarly compute $F_{m+1}(t)=D^{m+1}\Phi(P)$.

     We then factor $F_P(t)$. Each root $r$ of $F_P(t)$ determines a normalized $\ell$-isogeny $\phi_r$, and to compute that isogeny with {\rm\sf fastElkies} we need the codomain  $E_r$ and coefficient $\sigma_{\phi_r}$. 
     Factoring $F_P$ costs an expected $O(mn\M(mn))$ bit operations. 
     For each root $r$, having computed the normalized codomain $E_r$ and coefficient $\sigma_{\phi_r}$, the isogeny $\phi_r\colon E\to E_r$ can be computed with $O(\M(\ell))$ field operations with {\rm\sf fastElkies}~\cite{BMSS}. Our assumption that $\log e = O(\log p)$ implies that one inversion in $\FF_q$ costs $O(\M(n)n)$ bit operations, using the fast extended Euclidean Algorithm \cite[\S11.1]{moderncomputeralgebra} and Kronecker substitution \cite[\S8.4]{moderncomputeralgebra}. By batching the inversions in the subroutines in \cite[{\rm\sf fastElkies}, p.12]{BMSS}, the isogeny $\phi_r$ can be computed with $O(\M(\ell) \M(n)+n\M(n))$ bit operations. The cost of recovering all $m$ normalized isogenies in this way is $O(m(\M(\ell)\M(n)+n\M(n)))$. This is dominated by the cost of computing $F_P(t)$, so the total expected number of bit operations is $O(m\ell^2\M(n)+mn\M(mn))$. 
   
\end{proof}

\begin{theorem}\label{thm:complexity}
    Let $k=\FF_{p^e}$ be a finite field with $\log e = O(\log p)$ and $E$ an elliptic curve defined over $k$, and let $\ell$ be any prime satisfying $p>4\ell$. Assume $j(E)\not=0,1728$ and that $E$ is not $\ell$-isogenous to a curve of $j$-invariant $0$ or $1728$. Let $n=e\log p$. The normalized $\ell$-isogenies of $E$ can be computed in an expected 
    $O(
    \ell^3(\log \ell)^3
    +\ell n\M(\ell n) 
    + \ell^3\M(n) 
    )$
    bit operations.
\end{theorem}

\begin{proof}
    By work of Kunzweiler and Robert \cite[Theorem 5.2, Remark 5.3]{MR4837977} the modular polynomial $\Phi_{\ell}(X, Y) \in \ZZ[X, Y]$ can be computed in $O \left( \ell^3 \log(\ell)^3 \right)$ bit operations. Let $j=j(E)$. The polynomial $\varphi(Y)=\Phi_\ell(j,Y)$ has degree $\ell+1$ so completely factoring it takes an expected $O(\ell n \M(\ell n))$ bit operations~\cite{Rabin80}. By Lemma~\ref{lem:multipoint_complexity}, for each root $r$ of $\varphi(Y)$ in $k$, the (at most) $m_r$ $k$-rational $\ell$-isogenies corresponding to the point $(j,r)$ of multiplicity $m_r$ on $\Phi_\ell=0$ can be computed with $O(m_r\ell^2\M(n) 
    +m_rn\M(m_rn))$ bit operations. Since the multiplicities sum to $\ell+1$, computing all $k$-rational $\ell$-isogenies of $E$ costs 
    \[
    O\left(\sum_{r} m_r\ell^2\M(n)
    +m_rn\M(m_rn)\right) = O(\ell^3\M(n)+\ell n\M(\ell n))
    \]
    bit operations. Summing the complexities of the above steps gives the stated result. 
\end{proof}

\begin{remark}
    The quasi-cubic complexity of $\tilde{O}(\ell^3)$ in Theorem~\ref{thm:complexity} shows that asymptotically, using modular polynomials gives a factor of $\ell$ improvement over the algorithm of Tsukazaki~\cite[Chapter 3]{warwick57568}, which was previously the general algorithm used in SageMath in the \texttt{isogenies\_prime\_degree} algorithm. Recently, Lorenz Panny implemented our algorithm, reporting an appreciable improvement over the prior method (see SageMath pull requests \href{https://github.com/sagemath/sage/pull/42436}{\#42436} and \href{https://github.com/sagemath/sage/pull/42456}{\#42456}). 
\end{remark}

\section{Discussion and Examples}
\label{sec:examples}
In this section we discuss singular points of $\Phi_\ell=0$ and give several examples. SageMath code to reproduce these examples is available at \url{https://github.com/wmahaney/ModularMultipoints}.

\subsection{Bounding Multiplicity}
It is well known (see e.g. \cite[\textsection 7]{Schoof}) that for two ordinary elliptic curves $E_1, E_2$ over a field $k$ of appropriate type, so long as  $j(E_1),j(E_2)\neq 0,1728$,  there are at most two $\ell$-isogenies from $E_1$ to $E_2$ with distinct kernels. Thus the multiplicity of an ordinary point $(j_1,j_2)$ on $\Phi_\ell=0$ is at most $2$. 

First we give an ordinary example. 
\begin{example}
\label{ex:ordinary_double}
Take $p=37$, $\ell = 5$, and $k= \FF_p$. Set $j_1 = 2$ and $j_2 = 3$. Both $E_{j_1}, E_{j_2}$ are ordinary and $P=(j_1,j_2)$ is a double point on $\Phi_5(X,Y)=0$ over $\FF_{37}$. Fixing a model for the domain curve $E_{A,B}: y^2 = x^3 + 33x + 27$ the two codomain curves are given by 
\[
E_{A_{2_1}, B_{2_1}}: y^2 = x^3 + 23x + 14 , \quad E_{A_{2_2}, B_{2_2}}: y^2 = x^3 + 19x + 31
\] 
and the associated $\sigma$ values are 
\[
\sigma_1 = 25, \quad \sigma_2 = 32.
\]
The kernel polynomials are 
\[
f_1 = x^2 + 6x + 35 , \quad f_2 = x^2 + 21x + 29.
\]
\end{example}

If $E_1$ and $E_2$ are connected by 3 or more $\ell$-isogenies with pairwise distinct kernels, both $E_1$ and $E_2$ must be supersingular.
In this situation Kaneko's bound \cite[Theorem 2']{Kaneko} applied in $\End(E)$ can show $p < 4 \ell^4$; in 2025 Boscardin and Spindler improved this result to $p< 4 \ell^3$:

\begin{proposition} \cite[Corollary 16]{Improved_Kaneko} \label{prop:improvedkaneko}
If $p \neq \ell$ and $E_1, E_2$ are two elliptic curves over $\overline{\mathbb{F}}_p$ with $j(E) \neq 0$ and there exist at least three pairwise non-equivalent $\ell$-isogenies from $E_1\ra E_2$ then $p < 4 \ell^3$. 
\end{proposition}

Now we give an example of a supersingular multiple point of multiplicity greater than $2$.
\begin{example}
\label{ex:supersingular_quadruple}
Take $p = 29, \ell = 7$, and let $k = \FF_{p^2}= \FF_p(w)$ where $w$ has minimal polynomial $x^2 + 24x + 2$ over $\FF_p$. Take $j_1 = 2, j_2 = 25$ and fix $E_{A, B}: y^2 = x^3 + 3x + 28$. The four normalized codomain models are 
\begin{align*}
E_{A_{2_1}, B_{2_1}}: y^2 = x^3 + (19w+18)x + (17w+24)  , \quad  E_{A_{2_2}, B_{2_2}}: y^2 = x^3 +  (24w+17)x + (15w+11)  \\ 
 E_{A_{2_3}, B_{2_3}}: y^2 = x^3 + (5w+21)x + (14w+28)  , \quad E_{A_{2_4}, B_{2_4}}: y^2 = x^3 + (10w+26)x+(12w+22) 
\end{align*}
with $\sigma$ values
\begin{align*}
\sigma_1  = 11w+18  , \quad \sigma_2  = 13w+13 \\
\sigma_3  = 16w+20  , \quad \sigma_4  = 18w+15
\end{align*}
and kernel polynomials
\begin{align*}
f_1  = x^3 + (9w+20)x^2 + (15w+12)x + (7w+17)  , \quad f_2  = x^3 + (8w + 8)x^2 + (4w + 20)x + (14w + 16) \\
f_3  = x^3 + (21w + 19)x^2 + (25w + 11)x + (15w + 28)    , \quad f_4  = x^3 + (20w + 7)x^2 + 14wx + (22w + 23).
\end{align*}

\end{example}

\subsection{Rational Points on \texorpdfstring{$\Phi_{\ell}(X,Y)=0$}{Phi(X,Y)=0} Below Irrational Points on \texorpdfstring{$X_0(\ell)$}{X0l}}

Suppose $(j_1,j_2)$ is a $K$-rational point on $\Phi_{\ell}(X,Y)=0$ where $K$ is some number field or finite field. Let $(E,C)$ be a corresponding point on $X_0(\ell)$ and suppose further that $[E,C]\notin X_0(\ell)(K)$. We can assume $E$ is defined over $K$, since $j(E)=j(E_1)\in K$. This means $C$ is not $\Gal(\overline{K}/K)$-stable, and moreover, there is some $\sigma\in \Gal(\overline{K}/K)$ such that $C^{\sigma}\neq\mu(C)$ for any $\mu\in \Aut(E)$. Let $\phi\colon E\to E/C$ have kernel $C$; note we can take $E/C$ to also be defined over $K$ since $j(E/C)=j_2\in K$. And let $\phi^{\sigma}$ denote the isogeny defined  pointwise via $\phi^{\sigma}(P)= (\phi(P^{\sigma^{-1}}))^{\sigma}$. Then $\ker\phi^{\sigma}=C^{\sigma}$. We claim that $\widehat{\phi}\circ\phi^{\sigma}$ is a cyclic $\ell^2$-endomorphism of $E$. If not, then $\widehat{\phi}\circ\phi^{\sigma} = [\ell]\mu$ for some $\mu\in \Aut(E)$ but this would imply (via cancellation) that $\phi^{\sigma}=\phi\circ\mu$, contradicting our assumption that $C^{\sigma}\neq\mu(C)$ for any automorphism $\mu$. 
We collect the following 
\begin{theorem}\label{thm:irrationalisogeniesabovemultipoint}
    Suppose $(j_1,j_2)\in Z_{0}(\ell)(K)$ is not the image of any $K$-rational point in $X_0(\ell)(K)$. Then $(j_1,j_2)$ is a singular point on $Z_0(\ell)$ and $j_1$ has CM by an order of discriminant $t^2-4\ell^2$ for some $|t|<2\ell$. 
\end{theorem}

Such points exist. Consider any CM elliptic curve $E$ defined over a number field $L$ and suppose $E$ has (over $\overline{L}$) CM by an order $\OO$ containing a non-integer element $\alpha$ of norm $\ell^2$ in a quadratic imaginary field $K$. Suppose that $L$ does not contain $K$. Then $\End_L(E)=\mathbb{Z}$. The element $\alpha$ corresponds to a (geometric) cyclic endomorphism of $E$ of degree $\ell^2$ which therefore factors as a sequence of two $\ell$-isogenies $\alpha=\phi_2\circ\phi_1$. Then $\widehat{\phi_2},\phi_1$ have distinct kernels and cannot be defined over $L$ since otherwise $\alpha$ would be too. 

\begin{example}
\label{ex:number_field_fiber_poly_has_no_roots}
Consider an elliptic curve with CM by $\OO= \ZZ[\sqrt{-5}]\subseteq K= \QQ(\sqrt{-5})$. Such a curve has a $j$-invariant that is a root of $h(x)=x^2 - 1264000x - 681472000$, the Hilbert class polynomial for the discriminant $D=-20$. Let $L=\QQ[x]/h(x)$; then $L$ is linearly disjoint from $K$ (since $L$ is totally real, say). Thus if we take $E$ defined over $L$ with $j$-invariant $j_1$, a root of $h$ in $L$, then $\End_L(E)=\ZZ$. Also, $E$ has a cyclic $9$-endomorphism since $9=4+5=(2+\sqrt{-5})(2-\sqrt{-5})$. There are two distinct $3$-isogenies of $E$ (but not $L$-rational) to the curve with $j$-invariant $1264000-j(E)$ (the other root of $h$ in $L$). The polynomial $F_P$ has no roots in $L$ but does have two roots in $KL= L(\sqrt{-5})$.     
\end{example}

The next example shows that the rationality of the roots of the fiber polynomial $F_P$ for a singular point $P$ detects the rationality of the corresponding $\ell$-isogenies. In particular, we give another example of two elliptic curves over a number field $K$ with multiple (geometric) $\ell$-isogenies between them. In this example, the associated fiber polynomial $F_P$ does not have any roots in $K$, so the corresponding isogenies are not defined over $K$.

\begin{example}
\label{ex:ordinary_number_field_fiber_polynomial_irreducible}
    Consider an elliptic curve $E_1$ over $\QQ$ with geometric endomorphism ring $\ZZ[(1+\sqrt{-7})/2]$, e.g. with LMFDB label \href{https://www.lmfdb.org/EllipticCurve/Q/49a3/}{49.a2}. Then $j_1=j(E_1)=-3375$. Let $K=\QQ(\sqrt{-7})$ and $\OO_K$ be its ring of integers. Since $11$ splits in $\OO_K$ as 
    \[
    11 = (2+\sqrt{-7})(2-\sqrt{-7}),
    \]
    and since $2+\sqrt{-7}$ and $2-\sqrt{-7}$ are not unit multiples of one another, there are two distinct $11$-isogenies of $E_1$ to another curve with $j$-invariant $j_1$. Thus $P=(-3375,-3375)$ is an ordinary double point on $\Phi_{11}(X,Y)=0$. We fix a model of $E_1$ and then compute
    \[
    F_P(t) = 
    C (121 t^{2} - 182250 t + 922640625) 
    \]
    where 
    \[
    C = 2 \cdot 3^{66} \cdot 5^{30} \cdot 7^{13} \cdot 11 \cdot 13^{10} \cdot 17^{7} \cdot 19^{7} \cdot 31^{4} \cdot 41^{3} \cdot 59^{2} \cdot 61^{5} \cdot 73^{3}.
    \]
    The discriminant of $121t^2-182250t+922640625$ is 
    \[
    -1 \cdot 2^{6} \cdot 3^{10} \cdot 5^{6} \cdot 7
    \]
    so $F_P$ has no roots in $\QQ$ as expected; if it did, then $E$ would have two distinct $11$-isogenies defined over $\QQ$ and hence a cyclic $121$-endomorphism, but this is impossible since $\End(E)=\End_{\QQ}(E)=\ZZ$. But we see that $F_P$ has two roots in $K=\QQ(\sqrt{-7})$ corresponding to the two cyclic endomorphisms of degree $11$ in turn corresponding to the two prime divisors of $11$ in $\OO_K$. 
\end{example}

\section{AI Disclosure}
This paper was written without AI, but AI was used for proofreading  (Claude Opus 5.5 / Fable 5.2). 
AI was used for coding (ChatGPT 5 and Claude 5.5 models).

\section{Acknowledgments} 
We would like to thank Leonardo Constantin Mihalcea for helpful discussions. We would also like to thank the referees whose thoughtful feedback helped improve the paper. In particular, we thank the reviewer who provided the shorter, simpler argument that now appears in the proof of Lemma~\ref{lem:multi_normalized}. We would like to thank Lorenz Panny for his contribution of our algorithm to SageMath. The first author was supported by the Julian Chin Ph.D. Fellowship in Cybersecurity. This material is based upon work supported by the National Science Foundation under Grant No. 2340564. This research was funded in part by the Commonwealth of Virginia’s Commonwealth Cyber Initiative (CCI). The second author acknowledges support of the Institut Henri Poincaré (UAR 839 CNRS-Sorbonne Université) and LabEx CARMIN (ANR-10-LABX-59-01).

\bibliographystyle{alpha}
\bibliography{bib.bib}

\end{document}